\documentclass[12pt]{article}
\usepackage{amsmath,amssymb,url}

\usepackage{amsmath,amsthm,amssymb}
\usepackage{amscd}
\usepackage{amsfonts}
\usepackage{color}
\newcommand{\con}{\operatorname{Con}}

\newcommand{\non}{\operatorname{Non}}
\newcommand{\lin}{\operatorname{Lin}}

\newcommand{\var}{\operatorname{var}}

\newtheorem{theorem}{Theorem}[section]

\newtheorem{cor}[theorem]{Corollary}

\newtheorem{fact}[theorem]{Fact}
\newtheorem{lemma}[theorem]{Lemma}

\newtheorem{sufcon}{Sufficient Condition}
\newtheorem{claim}{Claim}

\begin{document}

\title{Lee monoid  $L_4^1$ is non-finitely based}
\author{ Inna A. Mikhailova\footnote{email: inna.mikhaylova@gmail.com, Ural Federal University, Ekaterinburg, Russia The first author was supported by the Russian Foundation for Basic Research, project
no.\ 17-01-00551, the Ministry of Education and Science
of the Russian Federation, project no.\ 1.3253.2017, and
the Competitiveness Program of Ural Federal University.
}\quad and Olga B. Sapir\footnote{email: olga.sapir@gmail.com}\\}
\date{}
\maketitle

\begin{abstract} We establish a new sufficient condition under which a monoid is non-finitely based and apply this condition to show that
the 9-element monoid  $L_4^1$ is non-finitely based. The monoid  $L_4^1$ was the only unsolved case in the finite basis problem for
 Lee monoids $L_\ell^1$, obtained by adjoining an identity element to the semigroup $L_\ell$ generated by two idempotents $a$ and $b$ subjected to the relation $0=abab \cdots$ (length $\ell$). We also prove a syntactic sufficient condition which is equivalent to the sufficient condition of Lee under which a semigroup is non-finitely based. This gives a new proof to the results of Zhang-Luo and Lee that the semigroup $L_\ell$ is non-finitely based for  each $\ell \ge 3$.

\vskip 0.1in


\end{abstract}

\section{Introduction}

An algebra  is said to be {\em finitely based} (FB) if there is a finite subset of its identities from which all of its identities may be deduced.
Otherwise, an algebra is said to be {\em non-finitely based} (NFB). By the celebrated McKenzie's result  \cite{RM}  the classes of FB and NFB finite algebras are recursively inseparable. It is still unknown whether the set of FB finite semigroups is recursive although a very large volume of work is devoted to this problem (see
the survey \cite{MV}).

Recently,  Lee suggested to investigate the finite basis property of the semigroups
\[
L_\ell = \langle a,b \mid aa=a, bb=b, \underbrace{ababab\cdots}_{\text{length }\ell} = 0 \rangle, \quad \ell \geq 2
\]
and the monoids $L_\ell^1$ obtained by adjoining an identity element to $L_\ell$.

The 4-element semigroup $L_2 =A_0$ is  long known to be finitely based \cite{1980}.
Zhang and  Luo proved \cite{WTZ1} that the 6-element semigroup $L_3$ is NFB
and Lee generalized this
result into a sufficient condition  \cite{EL} which implies that for all $\ell \geq 3$, the semigroup $L_\ell$ is NFB  \cite{EL1}.

 As for the monoids $L_\ell^1$, the 5-element monoid $L_2^1$ was also proved to be FB by Edmunds \cite{1977}, while the 7-element monoid $L_3^1$ is recently shown to be NFB by Zhang \cite{WTZ}.  It is proved in \cite{OS} that for each $\ell \ge 5$ the monoid $L_\ell^1$ is NFB. The goal of this article is to prove that  $L_4^1$ is NFB.

To this aim we establish a new sufficient condition under which a monoid is NFB.
Throughout this article, elements of a countably infinite alphabet $\mathfrak A$ are called {\em variables} and elements of the free monoid $\mathfrak A^*$  and free semigroup $\mathfrak A^+$ are called {\em words}.
We say that a word $\bf u$ has {\em the same type} as $\bf v$ if $\bf u$ can be obtained from $\bf v$ by changing the individual exponents of variables. For example, the words $x^2yxzx^5y^2xzx^3$ and
 $xy^2x^3zxyx^2zx$  are of the same type.

An {\em island} formed by a variable $x$  in a word {\bf u} is a maximal subword of $\bf u$ which is a power of $x$.  For example, the word $xyyx^5yx^3$ has  three islands formed by $x$ and two islands formed by $y$.
 We use $x^+$ to denote $x^n$ when $n$ is a positive integer and its exact value is unimportant.
If $\bf u$ is a word over a two-letter alphabet then the {\em height} of $\bf u$ is the number of islands in $\bf u$. For example,
the word $x^+$ has height 1,  $x^+y^+$ has height 2,   $x^+y^+x^+$ has height 3, and so on.
For each $\ell \ge 2$ consider the following property of a semigroup $S$.

$\bullet$ (C$_\ell$) If the height of ${\bf u} \in \{x,y\}^+$ is at most $\ell$, then
$\bf u$ can form an identity of $S$ only with a word of the same type.

 We use $\var S$ to refer to the variety of semigroups generated by $S$.
The following result from \cite{OS} gives us a connection between Lee semigroups, Lee monoids and Properties   (C$_\ell$).

\begin{fact} \label{final} \cite[Corollary 7.2]{OS} Let $\ell \ge 2$ and $S$ be a semigroup (resp. monoid).
Then $S$ satisfies Property (C$_\ell$) if and only if  $\var S$ contains  $L_{\ell}$ (resp. $L_{\ell}^1$).
\end{fact}

In view of Fact \ref{final}, the sufficient condition of Lee \cite{EL} (see Fact \ref{EL61} below) is equivalent to the following sufficient condition.

\begin{sufcon} \label{EL} (cf. Theorem \ref{EL62})
 Let $S$ be a semigroup that satisfies Property (C$_3$) and $k \ge 2$.
If for each $n \ge 2$, $S$ satisfies the identity

\begin{equation} \label{eL30}
x^ky^k_1y^k_2 \dots y^k_nx^k  \approx  x^ky^k_ny^k_{n-1} \dots y^k_1x^k\end{equation}
 then $S$ is NFB.
\end{sufcon}

Note that every monoid that satisfies \eqref{eL30} violates  Property (C$_2$). Therefore, Sufficient Condition \ref{EL} cannot be used to establish the non-finite basis property of any monoid. Theorem 2.7 in \cite{OS1} implies the result of Zhang  \cite{WTZ} that $L_3^1$ is NFB and can be reformulated as follows.

\begin{sufcon} \label{L3} \cite{OS1}
 Let $M$ be a monoid that satisfies Property (C$_3$).
If for each $n>0$, $M$ satisfies the identity
\begin{equation} \label{eL3}
(x_1 x_2 \dots x_{n-1}x_{n}) (y_1 y_2 \dots y_{n-1}y_n) (x_{n} x_{n-1} \dots x_2 x_1)  (y_n y_{n-1} \dots y_{2}y_1) \approx
\end{equation}
\[ (y_1 y_2 \dots y_{n-1}y_n) (x_1 x_2 \dots x_{n-1}x_{n})  (y_n y_{n-1} \dots y_{2}y_1) (x_{n} x_{n-1} \dots x_2 x_1),\]

then $M$ is NFB.
\end{sufcon}

Note that for $n=1$ the identity \eqref{eL3} fails on  $L_4^1$ and consequently on  $L_\ell^1$ for each $\ell \ge 4$.
Let $\pi$ denote the special permutation on $\{1, 2, \dots, n^2\}$ used by Jackson to prove Lemma 5.4 in \cite{MJ}.
The next theorem implies that  for each $\ell \ge 5$ the monoid $L_\ell^1$ is NFB \cite{OS}.

\begin{sufcon} \label{L5} \cite[Theorem 2.1]{OS}  Let $M$ be a monoid that satisfies Property (C$_5$).
If for each $n>3$, $M$ satisfies the identity
\begin{equation} \label{eL5}
(x_1 x_2 \dots x_{n^2-1}x_{n^2}) \hskip.04in  (x^k_{\pi1} x^k_{\pi2} \dots x^k_{\pi n^2})  \hskip.04in (x_{n^2} x_{n^2-1} \dots x_2 x_1) \approx
\end{equation}
\[(x_1 x_2 \dots x_{n^2-1}x_{n^2}) \hskip.04in  (x^k_{\pi n^2} \dots x^k_{\pi2} x^k_{\pi1} )     \hskip.04in ( x_{n^2} x_{n^2-1} \dots x_2 x_1)\]  for some $k \ge 1$, then  $M$ is NFB.

\end{sufcon}

Since $L_4^1$ satisfies $xyxyyx \approx xyxyxy$, it does not satisfy Property (C$_5$). Therefore, Sufficient Condition \ref{L5} cannot be used to establish the non-finite basis property of $L_4^1$.
The following theorem gives us a new sufficient condition under which a monoid is NFB  and will be proved in Section \ref{sec:thm}.

\begin{theorem} \label{main}

 Let $M$ be a monoid that satisfies Property (C$_4$).
If for each $n>0$, $M$ satisfies the identity
 \[{\bf U}_n = (x_1 x_2 \dots x_n)  (x_n x_{n-1} \dots x_1)  (x_1 x_2 \dots x_n)   \approx (x_1 x_2 \dots x_n)  (x_1^2 x_2^2 \dots x_n^2)  = {\bf V}_n,\]
then $M$ is NFB.
\end{theorem}

If $\tau$ is an equivalence relation on the free semigroup $\mathfrak A^+$ then we say that a word ${\bf u}$ is  a {\em $\tau$-term} for a semigroup $S$ if ${\bf u} \tau {\bf v}$ whenever $S$ satisfies ${\bf u} \approx {\bf v}$. Recall \cite{P} that $\bf u$ is an {\em isoterm}  for $S$ if ${\bf u} = {\bf v}$ whenever $S$ satisfies ${\bf u} \approx {\bf v}$. If $\bf u$ is an isoterm for $S$ then evidently, $\bf u$ is a $\tau$-term for $S$ for every equivalence relation $\tau$ on $\mathfrak A^+$.
It is shown in \cite{OS} that for $\ell \le 5$  the isoterms for $L_\ell^1$ carry no information about the non-finite basis property of $L_\ell^1$.
However, the non-finite basis property of Lee semigroups $L_\ell$ and Lee monoids $L^1_\ell$ for $\ell \ge 3$ can be established by analyzing $\tau$-terms, where $\tau$ is the equivalence relation on $\mathfrak A^+$ defined by ${\bf u} \tau {\bf v}$  if  ${\bf u}$ and $ {\bf v}$ are of the same type.

In particular, Sufficient Condition \ref{L5} is proved in \cite{OS} by  analyzing $\tau$-terms for monoids for which all words in $\{x,y\}^+$ of height at most $5$ are $\tau$-terms.
Likewise, we prove Theorem \ref{main}  by analyzing $\tau$-terms for monoids for which all words in $\{x,y\}^+$ of height at most $4$ are $\tau$-terms.
 In Section \ref{sect:sufconLee}, we prove  Sufficient Condition \ref{EL}  by analyzing $\tau$-terms for semigroups for which all words in $\{x,y\}^+$ of height at most $3$ are $\tau$-terms.

\section{Lee monoids  $L^1_\ell$ are NFB for all $\ell >2$}

We use $\con({\bf u})$ to denote the set of all variables contained in a word ${\bf u}$. Theorem \ref{main} implies the following.

\begin{cor} \label{L41}  The monoid $L_4^1 =  \langle a,b,1 \mid aa=a, bb=b, abab=0 \rangle$  is NFB.
\end{cor}

\begin{proof}  In view of Fact \ref{final}, it is enough to verify that
 $L_4^1$ satisfies the identity ${\bf U}_n \approx {\bf V}_n$ for each  $n>0$.

 Indeed, first notice that each variable appears at least $3$ times in ${\bf U}_n$ and ${\bf V}_n$.
Fix some substitution $\Theta: \mathfrak A \rightarrow L^1_4$. If for some $1 \le i \le n$, the set  $\con(\Theta(x_i))$ contains both $a$ and $b$ then both  $\Theta({\bf U}_n)$ and $\Theta({\bf V}_n)$ contain $(ab)^{3}$ or $(ba)^{3}$ as a subword and consequently, both are equal to zero. Therefore, we may assume that for each $1 \le i \le n$ we have $\Theta(x_i) \in \{a,b,1\}$.
To avoid some trivial cases we may also assume that $\Theta (x_1 x_2 \dots x_{n-1}x_{n})$ contains both letters $a$ and $b$. Consider two cases.

{\bf Case 1}: $\Theta (x_1 x_2 \dots x_{n-1}x_{n})$ contains $ab$ as a subword.

In this case, both $\Theta({\bf U}_n)$  and  $\Theta({\bf V}_n)$ contain $abab$ as a subword and consequently,
both are equal to zero.

{\bf Case 2}: $\Theta (x_1 x_2 \dots x_{n-1}x_{n}) = ba$.

In this case,  $\Theta({\bf U}_n) = (ba) (ab) (ba) = ba b a = \Theta({\bf V}_n)$.
\end{proof}

Notice that  ${\bf U}_3 \approx {\bf V}_3$ fails on $L_5^1$. Indeed, substitute $x_1 \rightarrow b$,  $x_2 \rightarrow a$, $x_3 \rightarrow b$.

\begin{theorem}  Lee monoid $L_\ell^1$ is FB if and only if $\ell=2$.

\end{theorem}

\begin{proof}
The 5-element monoid $L_2^1$ was proved to be FB by C. Edmunds \cite{1977}. The 7-element monoid $L_3^1$ is  NFB by the result of W. Zhang \cite{WTZ}. The 9-element monoid $L_4^1$ is NFB by Corollary \ref{L41}.
 For each $\ell \ge 5$ the monoid $L_\ell^1$ is NFB by the result of the second-named author \cite{OS}.
\end{proof}

\section{Identities of monoids that satisfy Property (C$_4$)}

 If a variable $t$ occurs exactly once in a word ${\bf u}$ then we say that $t$ is {\em linear} in ${\bf u}$. If a variable $x$ occurs more than once in  ${\bf u}$ then we say that $x$ is {\em non-linear} in ${\bf u}$.
 Evidently, $\con({\bf u}) = \lin({\bf u}) \cup \non({\bf u})$ where $\lin({\bf u})$ is the set of all linear variables in $\bf u$ and
$\non({\bf u})$ is the set of all non-linear variables in $\bf u$.
A {\em block} of $\bf u$ is a maximal subword of $\bf u$ that does not contain any linear variables of $\bf u$.

\begin{fact} \label{basic} \cite[Lemma 3.4]{OS} Let $M$ be a monoid that satisfies Property (C$_3$).
If $M \models {\bf u} \approx {\bf v}$ then

(i) $\lin({\bf u}) = \lin({\bf v})$,  $\non({\bf u}) = \non({\bf v})$  and the order of occurrences of linear letters in $\bf v$  is the same as in $\bf u$.

(ii) The corresponding blocks in $\bf u$ and $\bf v$ have the same content. In other words, if
\[{\bf u} = {\bf a}_0t_1 {\bf a}_1t_2 \dots t_{m-1}{\bf a}_{m-1} t_m {\bf a}_m,\] where $\non({\bf u}) = \con({\bf a}_0 {\bf a}_1 \dots {\bf a}_{m-1}{\bf a}_m)$ and  $\lin({\bf u}) = \{ t_1, \dots, t_m\}$, then
\[{\bf v} = {\bf b}_0t_1 {\bf b}_1t_2 \dots t_{m-1}{\bf  b}_{m-1} t_m {\bf b}_m\]
such that $\con({\bf a}_q) = \con({\bf b}_q)$ for each $0 \le q \le m$.

\end{fact}

If $\con({\bf u}) \supseteq \{x_1, \dots, x_n\}$ we write ${\bf u}(x_1, \dots, x_n)$ to refer to the word obtained from ${\bf u}$ by deleting all occurrences of all variables that are not in $\{x_1, \dots, x_n\}$ and say that $\bf u$ {\em deletes} to  ${\bf u}(x_1, \dots, x_n)$.

\begin{lemma} \label{2let} \cite[Lemma 3.6]{OS} Let $M$ be a monoid that satisfies Property (C$_4$).
Let ${\bf u}$  be a word with  $\non({\bf u}) =  \{x,y\}$ such that

(i)  ${\bf u}(x,y)$ has height at most $4$;

(ii) every block of $\bf u$  has  height  at most $3$.

Then $\bf u$  can form an identity of $M$ only with a word of the same type.
\end{lemma}

We use $D_x({\bf u})$ to denote the result of deleting all occurrences of variable $x$ in a word $\bf u$.
The next lemma is a special case of Lemma \ref{manylet}.

\begin{lemma} \label{3letter} Let $M$ be a monoid that satisfies Property (C$_4$).
Let ${\bf u}$  be a word with 3 non-linear variables such that

(I)   for each  $\{x,y\}\subseteq  \con({\bf u})$, the height of ${\bf u}(x,y)$ is at most $4$;

(II) no block of $\bf u$ deletes to $x^+ y^+ x^+ y^+$;

(III) no block of $\bf u$ deletes to $x^+ y^+ x^+ z^+ x^+$;

(IV) If some block $\bf B$ of $\bf u$ deletes to $x^+y^+z^+x^+$ then $\bf u$ satisfies each of the following:

(a) if there is an occurrence of $y$ to the  left of $\bf B$ then there is no occurrence of $z$ to the right of $\bf B$;

(b) if there is an occurrence of $z$ to the  left of $\bf B$ then there is no occurrence of $y$ to the right of $\bf B$.

Then $\bf u$  can form an identity of $M$ only with a word of the same type.
\end{lemma}

\begin{proof}  Suppose that $\non({\bf u}) = \{x,y,z\}$ and that  $M$ satisfies ${\bf u} \approx {\bf v}$.

In view of Fact \ref{basic}, in order to prove that ${\bf u}$ and ${\bf v}$ are of the same type,  it is enough to show that every block
$\bf B$ of $\bf u$ is of the same type as the corresponding block ${\bf B'}$ of ${\bf v}$.
In view of Condition (I)--(III), modulo duality and renaming variables there are only five possibilities for $\bf B$.

{\bf Case 1:}  $\bf B$ involves only $y$ and $z$.

In this case, the words $D_x({\bf u})$ and   $D_x({\bf v})$ are of the same type by  Lemma \ref{2let}. Therefore, the corresponding blocks  of  $D_x({\bf u})$ and   $D_x({\bf v})$  are also of the same type. In particular,
$\bf B$ is of the same type as the corresponding block $\bf B'$ of $\bf v$.

{\bf Case 2:}  ${\bf B} = x^+y^+z^+x^+$.

In this case, in view of Lemma \ref{2let}, we have ${\bf B'}(y,z) = y^+ z^+$,  ${\bf B'}(y,x) = x^+y^+ x^+$ and  ${\bf B'}(x,z) = x^+ z^+x^+$. So, if $\bf B'$ and $\bf B$ are not of the same type
then  ${\bf B'} = x^+y^+x^+z^+x^+$.

Modulo duality, there are four possibilities for the word $\bf u$.

{\bf Subcase 2.1:} Neither $y$ nor $z$ occurs in $\bf u$ to the  left of $\bf B$.

In this case, Condition (I) implies the following.
\begin{claim} \label{xyzx}

(a) the occurrences of $y$ can form at most one island (denoted by ${_2y}^+$ if any) to the right of $\bf B$;

(b) the occurrences of $z$ can form at most one island (denoted by ${_2z}^+$ if any) to the right of $\bf B$;

(c) the last occurrence of $x$ in $\bf u$ precedes  ${_2y}^+$ and  ${_2z}^+$.

\end{claim}

 Let $\Theta: \mathfrak A \rightarrow \mathfrak A^*$ be a substitution such that
$\Theta(y) = \Theta(z)=y$, $\Theta(x)=x$ and $\Theta(p)=1$ for each $p \not \in \{x,y\}$. Then $\Theta({\bf v})$ has height at least 5, but  in view of Claim \ref{xyzx}, $\Theta({\bf u})$ is either $x^+y^+x^+$  or  $x^+y^+x^+y^+$.
This contradicts Property (C$_4$).

{\bf Subcase 2.2:} Both $y$ and $z$ occur in $\bf u$ to the  left of $\bf B$.

In this case, Condition (IV) implies that  neither $y$ nor $z$ occurs in $\bf u$ to the  right of $\bf B$.
Consequently, this case is dual to Subcase 2.1.

{\bf Subcase 2.3:}  $y$ occurs in $\bf u$ to the  left of $\bf B$ but $z$ does not occur  in $\bf u$ to the  left of $\bf B$.

In this case, there is no $y$ to the right of $\bf B$ by  Condition (I), and no $z$ to the right of $\bf B$ by Condition (IV).
Thus, this case is also dual to Subcase 2.1.

{\bf Subcase 2.4:}  $z$ occurs in $\bf u$ to the  left of $\bf B$ but $y$ does not occur  in $\bf u$ to the  left of $\bf B$.

In this case, there is no $z$ to the right of $\bf B$ by  Condition (I), and no $y$ to the right of $\bf B$ by Condition (IV).
Thus, this case is also dual to Subcase 2.1.

{\bf Case 3:} ${\bf B} = x^+y^+z^+y^+x^+$.

In this case,  in view of Lemma \ref{2let}, we have ${\bf B'}(y,z) = y^+z^+y^+$ and  ${\bf B'}(y,x) = x^+ y^+ x^+$.
Therefore, $\bf B'$ must be of the same type as $\bf B$.

{\bf Case 4:} ${\bf B} = x^+y^+x^+z^+$

In this case,  in view of Lemma \ref{2let}, we have ${\bf B'}(x,y) = x^+ y^+x^+$ and  ${\bf B'}(x,z) = x^+z^+$.
Therefore, $\bf B'$ must be of the same type as $\bf B$.

{\bf Case 5:} ${\bf B} = x^+y^+z^+$

In this case,  in view of Lemma \ref{2let}, we have ${\bf B'}(x,y) = x^+ y^+$ and  ${\bf B'}(y,z) = y^+ z^+$.
Therefore, $\bf B'$ must be of the same type as $\bf B$.
\end{proof}

\begin{lemma} \label{3let} Two words $\bf u$ and $\bf v$ are of the same type
 if and only if $\con({\bf u}) = \con({\bf v})$ and for each set of three variables $\{x,y, z\} \subseteq \con({\bf u})$,
the words ${\bf u}(x,y,z)$ and ${\bf v} (x,y,z)$ are of the same type.
\end{lemma}

\begin{proof} If  $\bf u$ and $\bf v$ are of the same type then evidently,  $\con({\bf u}) = \con({\bf v})$ and for each
 $\mathfrak X \subseteq \con({\bf u})$ the words ${\bf u} (\mathfrak X)$ and ${\bf v} (\mathfrak X)$ are also of the same type.

Now suppose that   for each set of three variables $\{x,y, z\} \subseteq \con({\bf u})$,
the words ${\bf u}(x,y,z)$ and ${\bf v} (x,y,z)$ are of the same type.
Then $\bf u$ and $\bf v$ begin with the same letter.
If $\bf u$ and $\bf v$ are not of the same type then ${\bf u}= {\bf a}x y {\bf b}$ and ${\bf u}= {\bf a'}x z {\bf b'}$
for some possibly empty words $\bf a$, $\bf a'$, $\bf b$ and $\bf b'$
such that ${\bf a}x$ and ${\bf a'}x$ have the same type and $\{x,y,z\}$ are pairwise distinct variables.
Then the words ${\bf u}(x,y,z)$ and  ${\bf v}(x,y,z)$ are also not of the same type. To avoid a contradiction, we must assume
that $\bf u$ and $\bf v$ are of the same type.
\end{proof}

\begin{lemma} \label{manylet} Let $M$ be a monoid that satisfies Property (C$_4$).
Let ${\bf u}$  be a word  such that

(I)  for each  $\{x,y\}\subseteq  \con({\bf u})$, the height of  ${\bf u}(x,y)$ is at most $4$;

(II) no block of $\bf u$ deletes to $x^+ y^+ x^+ y^+$;

(III) no block of $\bf u$ deletes to $x^+ y^+ x^+ z^+ x^+$;

(IV) If some block $\bf B$ of $\bf u$ deletes to $x^+y^+z^+x^+$ then $\bf u$ satisfies each of the following:

(a) if there is an occurrence of $y$ to the  left of $\bf B$ then there is no occurrence of $z$ to the right of $\bf B$;

(b) if there is an occurrence of $z$ to the  left of $\bf B$ then there is no occurrence of $y$ to the right of $\bf B$.

Then $\bf u$  can form an identity of $M$ only with a word of the same type.
\end{lemma}

\begin{proof}   Suppose that  $M$ satisfies ${\bf u} \approx {\bf v}$.
Let $\bf B$ and $\bf B'$ be the corresponding blocks in $\bf u$ and $\bf v$. Then for each  $\{x,y, z\} \subseteq \con({\bf u})$
the words ${\bf B}(x,y,z)$ and  ${\bf B'}(x,y,z)$  are of the same type by Lemma \ref{3letter}. Therefore, the words ${\bf B}$ and  ${\bf B'}$  are also of the same type by Lemma \ref{3let}.
\end{proof}

\section{Properties of words applicable to  ${\bf U}_n$}

\begin{fact} \label{Eu}   \cite[Fact 4.2]{OS} Given a word $\bf u$ and a substitution $\Theta: \mathfrak A \rightarrow \mathfrak A^+$, one can rename some variables in $\bf u$ so that the resulting word $E({\bf u})$ has the following properties:

(i)  $\con (E({\bf u})) \subseteq  \con ({\bf u})$;

(ii) $\Theta(E({\bf u}))$ is of the same type as  $\Theta ({\bf u})$;

(iii) for every  $x, y \in \con (E({\bf u}))$,  if the words $\Theta(x)$  and  $\Theta(y)$ are powers of the same variable then $x=y$.

\end{fact}

\begin{lemma} \label{prop1} \cite[Lemma 4.3]{OS} Let $\bf U$ be a word such that for each  $\{x,y\}\subseteq  \con({\bf U})$,
 the height of ${\bf U}(x,y)$ is at most $4$.
 Let $\Theta: \mathfrak A \rightarrow \mathfrak A^+$  be a substitution which satisfies Property (ii) in Fact \ref{Eu}.
If $\Theta({\bf u}) = {\bf U}$ then $\bf u$  satisfies Condition (I)  in Lemma \ref{manylet}, that is,
 for each  $\{x,y\}\subseteq  \con({\bf u})$,  the height of ${\bf  u}(x,y)$ is at most $4$.

\end{lemma}

A word ${\bf u}$ is called a {\em scattered subword} of a word ${\bf v}$ whenever there exist words ${\bf u}_1, \dots, {\bf u}_k, {\bf v}_0, {\bf v}_1, \dots, , {\bf v}_{k-1}, {\bf v}_k \in \mathfrak A^*$ such that
${\bf u} = {\bf u}_1 \dots {\bf u}_k$ and ${\bf v} = {\bf v}_0 {\bf u}_1 {\bf v}_1\dots {\bf v}_{k-1}{\bf u}_k{\bf v}_k;$
in other terms, this means that one can extract $\bf u$ treated as a sequence of letters from the sequence $\bf v$. For example,
$x_1x_3$ is a scattered subword of $x_1x_2x_3x_2x_1$.

For the rest of this section, for each $n>2$, we use ${\bf U}_n$ to denote a word of the same type as  $(x_1 x_2 \dots x_n)  (x_n x_{n-1} \dots x_1)  (x_1 x_2 \dots x_n) $. The following properties of ${\bf U}_n$ can be easily verified:

$\bullet$ (P1)  ${\bf U}_n(x_i, x_j) = x_i^+ x_j^+ x_i^+x_j^+$  for each $1 \le i < j \le n$;

 $\bullet$ (P2)  ${\bf U}_n(x_i,x_j,x_k)=x_i^+x_j^+x_k^+ x_j^+ x_i^+x_j^+x_k^+$  for each $1 \le i < j < k \le n$;

$\bullet$ (P3)  $x_ix_j$ appears exactly twice in ${\bf U}_n$ as a scattered subword and
$x_j x_i$  appears exactly once in ${\bf U}_n$ as a scattered subword  for each $1 \le i < j \le n$;

$\bullet$ (P4)  ${\bf U}_n$ contains $x^+_n x^+_{n-1} \dots x^+_2 x^+_1$  as a subword between the two scattered  subwords $x_i x_j$
 for each $1 \le i < j \le n$;

$\bullet$ (P5)  the occurrences of $x_1$ and $x_n$ form exactly two islands in ${\bf U}_n$ and for each $1 <i<n$, the occurrences of $x_i$ form exactly three islands in  ${\bf U}_n$. We refer to these  islands as  ${_1x_{1}^+}$,   ${_2x_{1}^+}$,  ${_1x_{n}^+}$,   ${_2x_{n}^+}$,   ${_1x_{i}^+}$, ${_2x_{i}^+}$ and  ${_3x_{i}^+}$;

$\bullet$ (P6)  for each $1 < i < n$, ${\bf U}_n$
contains $x^+_n x^+_{n-1} \dots x^+_{i+2} x^+_{i+1}$ as a subword  between  ${_1x_{i}^+}$  and  ${_2x_{i}^+}$  and, contains
  $x^+_{i-1} \dots x^+_1$ as a subword  between  ${_2x_{i}^+}$  and  ${_3x_{i}^+}$. Also, ${\bf U}_n$
contains $x^+_n x^+_{n-1} \dots  x^+_{2}$ as a subword  between  ${_1x_{1}^+}$  and  ${_2x_{1}^+}$  and, contains
  $x^+_{n-1} \dots x^+_1$ as a subword  between  ${_1x_{n}^+}$  and  ${_2x_{n}^+}$;

$\bullet$ (P7)  for each $1 < i \ne j < n$, if $x_j$ occurs in ${\bf U}_n$ to the left of  ${_1x_{i}^+}$ or between ${_2x_{i}^+}$  and  ${_3x_{i}^+}$ then $j<i$;  if $x_j$ occurs in ${\bf U}_n$ between ${_1x_{i}^+}$  and  ${_2x_{i}^+}$ or  to the right of  ${_3x_{i}^+}$
then $j>i$.

\begin{lemma} \label{propii}  Let $\bf u$ be a word in less than $n-1$ variables  for some $n>2$. Let
 $\Theta: \mathfrak A \rightarrow \mathfrak A^+$  be a substitution  which satisfies Property (ii) in Fact \ref{Eu}.

If $\Theta({\bf u})={\bf U}_n$, then  ${\bf u}$ satisfies Condition (II)  in Lemma \ref{manylet}, that is,  no block of $\bf u$ deletes to $x^+ y^+ x^+ y^+$.
\end{lemma}
\begin{proof}
  Suppose that some block $\bf B$ of ${\bf u}$ deletes to $x^+y^+x^+y^+$, where $\{x,y\}\subseteq \con({\bf u})$. In view of Fact~\ref{Eu}, there are $1\leq i\neq j\leq n$ such that $x_i\in \con(\Theta(x))$ and $x_j \in\con(\Theta(y))$. Since  $\Theta({\bf B})$ contains $x_i x_j$ twice as a scattered subword,   Property (P3) implies that $i < j$.
Then $\Theta({\bf B})$ contains $x^+_n x^+_{n-1} \dots x^+_2 x^+_1$  between the two scattered  subwords $x_i x_j$ by  Property (P4).

Since $\con({\bf u})$ involves less than $n-1$ distinct variables there is a variable $t\in\con({\bf B})$ such that
 $\Theta(t)$ contains $x_{k+1} x_{k}$ for some $1 < k <n$. In view of Property (P3),
 $t$ is linear in $\bf u$. Therefore, there is a linear letter between the two scattered subwords $xy$ in $\bf B$.
A contradiction.
\end{proof}

\begin{lemma} \label{propiii}   Let $\bf u$ be a word in less than $n-1$ variables  for some $n>2$. Let
 $\Theta: \mathfrak A \rightarrow \mathfrak A^+$  be a substitution  which satisfies Property (ii) in Fact \ref{Eu}.

If $\Theta({\bf u})={\bf U}_n$, then  ${\bf u}$ satisfies Condition (III)  in Lemma \ref{manylet}, that is, no block of $\bf u$ deletes to $x^+ y^+ x^+ z^+ x^+$.
\end{lemma}
\begin{proof}
  Suppose that some block $\bf B$ of  ${\bf u}$ deletes to $x^+y^+x^+z^+x^+$, where $x,y,z$ are three distinct non-linear variables that belong to $\con({\bf u})$.

 Due to Fact~\ref{Eu}, there are pairwise distinct $1\leq i,j,k\leq n$ such that $x_i\in\con(\Theta(x))$, $x_j\in\con(\Theta(y))$ and $x_k\in\con(\Theta(z))$. Therefore, the occurrences of $x_i$ form at least three islands in $\Theta({\bf B})$. By Property (P5), $1<i<n$
and the occurrences of $x_i$ form exactly three islands in $\Theta({\bf B})$. Property (P6) implies that $\Theta({\bf B})$ contains $x^+_n x^+_{n-1} \dots x^+_{i+2} x^+_{i+1}$  between the first two islands formed by $x_i$
and  $x^+_{i-1} \dots x^+_1$  between the last two islands formed by $x_i$.

Since $\con({\bf u})$ involves less than $n-1$ distinct variables there is a variable $t\in\con({\bf B})$ such that
 $\Theta(t)$ contains $x_{r+1} x_{r}$ for some $1 \le  r <  i-1$ or $i+1 \le r < n$. In view of Property (P3),
 $t$ is linear in $\bf u$. Therefore, there is a linear letter in $\bf B$ either between the first two islands formed by $x$ or between the last two islands formed by $x$.
A contradiction.
\end{proof}

\begin{lemma} \label{prop2}  Let $\bf u$ be a word in less than $n-1$ variables  for some $n>2$ such that some block $\bf B$ of $\bf u$
deletes to $x^+y^+z^+x^+$. Let
 $\Theta: \mathfrak A \rightarrow \mathfrak A^+$  be a substitution  which satisfies Property (ii) in Fact \ref{Eu}.

If $\Theta({\bf u})={\bf U}_n$, then  ${\bf u}$ satisfies Condition (IV)  in Lemma \ref{manylet}, that is, each of the following holds:

(a) if there is an occurrence of $y$ to the  left of $\bf B$ then there is no occurrence of $z$ to the right of $\bf B$;

(b) if there is an occurrence of $z$ to the  left of $\bf B$ then there is no occurrence of $y$ to the right of $\bf B$.

\end{lemma}

\begin{proof} If $\Theta(x)$ is not a power of a variable, then  $\Theta(x)$ contains $x_kx_{k+1}$ for some $1 \le k <  n$ and  $x$ appears twice in $\bf u$ by Property (P3). Since $\bf u$ involves less than $n-1$ variables, there is a linear letter between the two occurrences of $x$ in $\bf u$ due to Property (P4). This contradicts the fact that $\bf B$ is a block of $\bf u$. So, we can assume
that   $\Theta(x) = x_i^+$ for some $1 \le i \le n$.

Due to Property (P5), the occurrences of $x_i$ form at most three islands in ${\bf U}_n$.  We refer to the two islands formed by $x$ in $\bf B$ as ${_1x^+}$ and   ${_2x^+}$.
Since $\Theta$ satisfies  Property (ii) in Fact \ref{Eu}, four cases are possible.

{\bf Case 1:} $i=1$ or $i=n$, $\Theta({_1x^+})$ is a subword of  ${_1x_{i}^+}$ and  $\Theta({_2x^+})$ is a subword of  ${_2x_{i}^+}$.

 Since $\bf u$ involves less than $n-1$ variables, there is a linear letter between the two occurrences of $x$ in $\bf u$ due to Property (P6).
This contradicts the fact that $\bf B$ is a block of $\bf u$.

{\bf Case 2:} $1 < i < n$, $\Theta({_1x^+})$ is a subword of  ${_1x_{i}^+}$ and  $\Theta({_2x^+})$ is a subword of  ${_3x_{i}^+}$.

Use the same arguments as for Case 1.

{\bf Case 3:} $1 < i < n$, $\Theta({_1x^+})$ is a subword of  ${_1x_{i}^+}$ and  $\Theta({_2x^+})$ is a subword of  ${_2x_{i}^+}$.

In this case, in view of Property (P7), neither $y$ nor $z$ occurs to the left of $\bf B$.

{\bf Case 4:} $1 < i < n$, $\Theta({_1x^+})$ is a subword of  ${_2x_{i}^+}$ and  $\Theta({_2x^+})$ is a subword of  ${_3x_{i}^+}$.

In this case, in view of Property (P7), neither $y$ nor $z$ occurs to the right of $\bf B$.
\end{proof}

\section{Proof of Theorem \ref{main}}\label{sec:thm}

\begin{lemma} \label{nfblemma}  \cite[Lemma 5.1]{OS}  Let $\tau$ be an equivalence relation on the free semigroup $\mathfrak A^+$ and $S$ be a semigroup.
Suppose that for infinitely many $n$, $S$ satisfies an identity ${\bf U}_n \approx {\bf V}_n$ in at least $n$ variables
such that ${\bf U}_n$ and ${\bf V}_n$ are not $\tau$-related.

Suppose also that for every identity ${\bf u} \approx {\bf v}$ of $S$ in less than $n-1$ variables, every
 word  $\bf U$ such that ${\bf U} \tau {\bf U}_n$ and every substitution
 $\Theta: \mathfrak A \rightarrow \mathfrak A^+$ such that $\Theta({\bf u}) = {\bf U}$ we have
 ${\bf U} \tau \Theta({\bf v})$.  Then $S$ is NFB.
\end{lemma}

We use $\con_{2}({\bf u})$ to denote the set of all variables which occur twice  in $\bf u$ and $\con_{>2}({\bf u})$ to denote the set of all
variables which occur at least 3 times in $\bf u$. The next lemma is similar to Lemma 4.1 in \cite{OS} (Lemma \ref{redef1} below).

\begin{lemma} \label{redef}  Let $\bf u$ and $\bf v$ be two words of the same type such that $\lin({\bf u}) = \lin({\bf v})$,
$\con_2({\bf u}) = \con_2({\bf v})$ and $\con_{>2}({\bf u}) = \con_{>2}({\bf v})$.

Let  $\Theta: \mathfrak A \rightarrow \mathfrak A^+$
be a substitution that has the following property:

(*) If $\Theta(x)$ contains more than one variable then $x \in \lin({\bf u}) \cup \con_{2}({\bf u})$.

Then $\Theta({\bf u})$ and  $\Theta({\bf v})$  are also of the same type.

\end{lemma}

\begin{proof} Since $\bf u$ and $\bf v$ are of the same type, the following is true.

\begin{claim} \label{yy} Suppose that $y \in \con_2({\bf u})$. If there is an occurrence of $x$ between the two occurrences of $y$  in $\bf u$  then there is an occurrence of $x$ between the two occurrences of $y$  in $\bf v$.
\end{claim}

Since $\bf u$ and $\bf v$ are of the same type, for some $r \ge 1$ and $u_1, \dots,u_r, v_1, \dots, v_r >0$ we have ${\bf u} =  c_1^{u_1} c_2^{u_2}\dots c_r^{u_r}$ and  ${\bf v} =  c_1^{v_1}c_2^{v_2}\dots  c_r^{v_r}$,
where $c_1, \dots , c_r \in \mathfrak A$ are such that $c_i \ne c_{i+1}$.

 First, let us prove that for each $1 \le i \le r$ the words $\Theta({c}_i^{u_i})$ and $\Theta({c}_i^{v_i})$ are of the same type.
Indeed, If $c_i$ is linear in $\bf u$ (and in $\bf v$) then $u_i=v_i=1$ and $\Theta({c}_i^{u_i}) =  \Theta({c}_i^{v_i})$.
If $c_i$ occurs twice in  $\bf u$ (and in $\bf v$), then in view of Claim \ref{yy}, either  $u_i=v_i=1$ or  $u_i=v_i=2$.
In either case, we have  $\Theta({c}_i^{u_i}) =  \Theta({c}_i^{v_i})$.
If $c_i$ occurs at least 3 times in $\bf u$  (and in $\bf v$) then $\Theta({c}_i^{u_1}) = x^+$ for some variable $x$ and  $\Theta({c}_i^{v_i})$ is a power of the same variable.

Since \[\Theta({\bf u}) = \Theta({c}_1^{u_1}{c}_2^{u_2}\dots {c}_r^{u_r} ) = \Theta({c}_1^{u_1}) \Theta({c}_2^{u_2})\dots \Theta({c}_r^{u_r})\] and
\[\Theta({\bf v}) = \Theta({c}_1^{v_1}{c}_2^{v_2}\dots {c}_r^{v_r} ) = \Theta({c}_1^{v_1}) \Theta({c}_2^{v_2})\dots \Theta({ c}_r^{v_r}),\]
we conclude that $\Theta({\bf u})$ and $\Theta({\bf v})$ are of the same type.
\end{proof}

\begin{proof}[Proof of Theorem \ref{main}]  Let $\tau$ be the equivalence relation on $\mathfrak A^+$ defined by ${\bf u} \tau {\bf v}$ if $\bf u$ and $\bf v$ are of the same type.
 First, notice that the words  ${\bf U}_n$ and  ${\bf V}_n$ are not of the same type. Indeed, ${\bf U}_n$
contains $x_{n} x_{n-1}$ as a subword but ${\bf V}_n$ does not have this subword.

Now let  ${\bf U}$ be of the same type as ${\bf U}_n$.
Let ${\bf u} \approx {\bf v}$ be an identity of $M$ in less than $n-1$ variables  and let
 $\Theta: \mathfrak A \rightarrow \mathfrak A^+$  be a substitution such that $\Theta({\bf u}) = {\bf U}$.
Notice that $E({\bf u})$ also involves less than $n-1$ variables and $E({\bf u}) \approx E({\bf v})$ is also an identity of $M$.

Due to Property (P1), for each  $1\le i < j \le n$ the height of  ${\bf U}(x_i, x_j)$ is at most $4$.
 So, by Lemma \ref{prop1}, $E({\bf u})$ satisfies Condition (I)  in Lemma  \ref{manylet}, that is, for each  $\{x,y\}\subseteq  \con(E({\bf u}))$ the height of $E({\bf u}(x,y))$ is at most 4. Also,   $E({\bf u})$ satisfies Conditions (II)--(IV)  in Lemma  \ref{manylet} by Lemmas \ref{propii}--\ref{prop2}.

 Therefore, Lemma \ref{manylet} implies that the word $E({\bf v})$ is  of the same type as $E({\bf u})$.
Due to Property (P3), for each $1 \le i \ne j \le n$ the word $x_i x_j$ appears at most twice in $\bf U$ as a scattered subword.
Consequently, the word $E({\bf u})$ and the substitution $\Theta$ satisfy Condition (*) in  Lemma \ref{redef}. According to Fact 3.1 in \cite{OS}, Property (C$_4$) implies that the word  $x^2$ is an isoterm for $M$. Thus
all other conditions of Lemma  \ref{redef} are also met.
Therefore, the word $\Theta(E({\bf v}))$ has the same type as  $\Theta(E({\bf u}))$ by Lemma \ref{redef}.
Thus we have
\[ {\bf U} = \Theta({\bf u}) \stackrel{Fact \ref{Eu}}{\tau} \Theta (E ({\bf u})) \stackrel{Lemma \ref{redef}} {\tau} \Theta (E ({\bf v})) \stackrel{Fact \ref{Eu}}{\tau} \Theta({\bf v}).\]
Since $\Theta({\bf v})$ is of the same type as $\bf U$, $M$ is NFB by Lemma \ref{nfblemma}.
\end{proof}

\section{Syntactic version of the sufficient condition of Lee for semigroups}\label{sect:sufconLee}

The following sufficient condition implies that for each $\ell \ge 3$ the semigroup $L_\ell$ is NFB \cite{EL1}.

\begin{fact} \label{EL61} \cite[Theorem 11]{EL} Fix $k \ge 2$.
 Let $S$ be a semigroup such that $\var S$ contains $L_3$.
If for each $n \ge 2$, $S$ satisfies the identity

\[x^ky^k_1y^k_2 \dots y^k_nx^k  \approx  x^ky^k_ny^k_{n-1} \dots y^k_1x^k\]

 then $S$ is NFB.
\end{fact}

In view of Fact \ref{final},  the following sufficient condition is a slight generalization of  Fact  \ref{EL61}.

\begin{theorem} \label{EL62}
 Let $S$ be a semigroup that satisfies Property (C$_3$).
If for infinitely many  $n\ge 2$, $S$ satisfies the identity

\[{\bf U}_n= x^ky^k_1y^k_2 \dots y^k_nx^k  \approx  x^ky^k_ny^k_{n-1} \dots y^k_1x^k ={\bf V}_n\]

for some $k \ge 2$, then $S$ is NFB.
\end{theorem}

The goal of this section is to prove Theorem \ref{EL62} directly using Lemma \ref{nfblemma}.
To this aim, we establish some consequences of Property (C$_3$) for semigroups.

\begin{lemma} \label{sem1} Let $S$ be a semigroup that satisfies Property (C$_3$).
If $S \models {\bf u} \approx {\bf v}$ then
$\con({\bf u}) = \con({\bf v})$ and  $\lin({\bf u}) = \lin({\bf v})$.

\end{lemma}

\begin{proof} If $x \in \con({\bf u})$ but  $x \not \in \con({\bf v})$ then for some $y \in \con({\bf v})$ and $r>0$
the identity ${\bf u} \approx {\bf v}$ implies $y^r \approx {\bf w}$ such that the height of ${\bf w} \in \{x,y\}^+$ is at least 2.
To avoid a contradiction to Property (C$_1$) we conclude that $\con({\bf u}) = \con({\bf v})$.

If $x$ is linear in $\bf u$ but appears at least twice in $\bf v$ then substitute $xy$ for $x$ and $y$ for all other variables.
Then for some $c+d >0$ the identity ${\bf u} \approx {\bf v}$ implies $y^c x y^d \approx {\bf w}$ such that the height of ${\bf w} \in \{x,y\}^+$ is at least 4. To avoid a
contradiction to Property (C$_3$) we conclude that $\lin({\bf u}) = \lin({\bf v})$.
\end{proof}

\begin{lemma} \label{sem2} Let $S$ be a semigroup that satisfies Property (C$_3$).
If each variable forms only one island in a word $\bf u$ then $\bf u$ can form an identity of $S$ only with a word of the same type.
\end{lemma}

\begin{proof}
   Since each variable forms only one island in the word $\bf{u}$ we may assume that ${\bf u}=x^+_1x^+_2\ldots x^+_r$ for some distinct variables $x_1,x_2, \ldots, x_r$. Suppose that $\bf{u}$ forms an identity of $S$ with a word $\bf{v}$. By Lemma \ref{sem1}, $\con(\bf{u})=\con(\bf{v})$. Note that if we replace $x_1$ by $x$ and any other letter by $y$, then the word $\bf{u}$ turns into $x^+y^+$.
   Since $S$ satisfies Property (C$_2$) the word $\bf{v}$ also starts with $x_1$.

   First, let us prove that  each variable forms only one island in the word $\bf{v}$. Suppose the contrary, there is a letter $a\in \con(\bf{u})=\con(\bf{v})$ that forms at least two islands in $\bf{v}$.  We replace $a$ by $y$ and other letters by $x$. Thus, the identity $\bf{u}\approx \bf{v}$ implies $\bf{u}'\approx \bf{v}'$.
   Notice that the height of ${\bf u'} \in \{x,y\}^+$ is at most $3$ and that $y$ forms only one island in $\bf u'$. On the other hand,
   $y$ forms at least two islands in $\bf v'$. Since $S$ satisfies Property (C$_3$), this is impossible.

  Second, let us take two consecutive letters $x_i, x_{i+1}\in\con(\bf{u})$ and replace them by $y$, while other letters by $x$. Note that the word $\bf{u}$ transforms into a word ${\bf u}'\in\{x,y\}^+$ of height at most $3$ with one island of $y$. Therefore, by Property (C$_3$) the corresponding word $\bf{v}'$ also has one island of $y$ implying the word $\bf{v}$ has  either $x_ix_{i+1}$ or $x_{i+1}x_i$ as a  subword. Since $\bf{u}$ and $\bf{v}$ start with the same letter $x_1$ we conclude that $\bf{u}$ and $\bf{v}$ have the same type.
\end{proof}

\begin{lemma} \label{sem3} Let $S$ be a semigroup that satisfies Property (C$_3$). Let $\bf u$ be a word
that begins and ends with $x$ such that $x$ forms exactly two islands in $\bf u$.
Suppose also that $\bf u$ contains a linear letter and that each variable other than $x$ forms only one island in $\bf u$.
Then $\bf u$ can form an identity of $S$ only with a word of the same type.
\end{lemma}

\begin{proof} Suppose that $S \models {\bf u} \approx {\bf v}$.
If we substitute $y$ for all variables in $\con({\bf u}) = \con({\bf v})$ other than $x$ then $\bf u$ turns into $x^+y^+x^+$.
Since $S$ satisfies Property (C$_3$), this implies that $\bf v$ starts and ends with $x$ and $x$ forms exactly two islands in $\bf v$.
Thus we have ${\bf u} = x^+ {\bf a} t {\bf b} x^+$ and ${\bf v} = x^+ {\bf a'} t {\bf b'} x^+$ where $\con ({\bf ab})= \con ({\bf a'b'}) \cap \{x,t\} =\emptyset$.

We have ${\bf b} = y_1^+ y_2^+ \dots y_r^+$ for some $\{y_1, \dots, y_r\} \subseteq \con ({\bf ab})$.
If $y_1$ is not the first letter in ${\bf b'}$ then substitute $xy$ for $t$, $y$ for $y_1$ and $x$ for all other variables.
Then $\bf u$ turns into $x^+ y^+ x^+$ while $\bf v$ becomes a word that contains at least two islands of $y$. To avoid a contradiction
to Property (C$_3$) we conclude that $\bf b'$ starts with $y_1$. If $\bf b'$ does not have $y_1y_2$ as a subword, then substitute $xy$ for $t$, $y$ for $y_1$ and for $y_2$ and $x$ for all other variables. Then $\bf u$ turns into $x^+ y^+ x^+$ while $\bf v$ becomes a word that contains at least two islands of $y$. To avoid a contradiction
to Property (C$_3$) we conclude that $\bf b'$ starts with $y_1^+y_2$. And so on. Eventually we conclude
that the words $\bf b$ and $\bf b'$ are of the same type.

In a similar way, one can show that $\bf a$ and $\bf a'$ are also of the same type. Consequently, $\bf u$ and $\bf v$ are of the same type.
\end{proof}

Finally, in order to prove Theorem \ref{EL62} we need the following statement similar to Lemma \ref{redef}.

\begin{lemma} \label{redef1} \cite[Lemma 4.1]{OS} Let $\bf u$ and $\bf v$ be two words of the same type such that $\lin({\bf u}) = \lin({\bf v})$ and
$\non({\bf u}) = \non({\bf v})$.

Let  $\Theta: \mathfrak A \rightarrow \mathfrak A^+$
be a substitution that has the following property:

(*) If $\Theta(x)$ contains more than one variable then $x$ is  linear  in $\bf u$.

Then $\Theta({\bf u})$ and  $\Theta({\bf v})$  are also of the same type.

\end{lemma}

\begin{proof}[Proof of Theorem \ref{EL62}]  Let $\tau$ be the equivalence relation on $\mathfrak A^+$ defined by ${\bf u} \tau {\bf v}$ if $\bf u$ and $\bf v$ are of the same type.
 First, notice that the words  ${\bf U}_n$ and  ${\bf V}_n$ are not of the same type. Indeed, ${\bf U}_n$
contains $y_1 y_2$ as a subword but ${\bf V}_n$ does not have this subword.

Now let  ${\bf U}$ be of the same type as ${\bf U}_n$.
Let ${\bf u} \approx {\bf v}$ be an identity of $S$ in less than $n$ variables  and let
 $\Theta: \mathfrak A \rightarrow \mathfrak A^+$  be a substitution such that $\Theta({\bf u}) = {\bf U}$.
Notice that $E({\bf u})$ also involves less than $n$ variables and $E({\bf u}) \approx E({\bf v})$ is also an identity of $S$.

If every variable forms only one island in $E({\bf u})$ then by Lemma \ref{sem2} the word $E({\bf v})$ is of the same type as $E({\bf u})$.
If some variable $x$ forms more than one island in $E({\bf u})$ then in view of Fact \ref{Eu}, $x$ forms exactly two islands in $E({\bf u})$ and
$E({\bf u})$ begins and ends with $x$. Also, $E({\bf u})$ contains a linear letter because it involves less than $n$ variables.
So, in this case, the word $E({\bf v})$ is of the same type as $E({\bf u})$ by Lemma \ref{sem3}.

Therefore, the word $\Theta(E({\bf v}))$ has the same type as  $\Theta(E({\bf u}))$ by Lemma \ref{redef1}.
Thus we have
\[ {\bf U} = \Theta({\bf u}) \stackrel{Fact \ref{Eu}}{\tau} \Theta (E ({\bf u})) \stackrel{Lemma \ref{redef1}} {\tau} \Theta (E ({\bf v})) \stackrel{Fact \ref{Eu}}{\tau} \Theta({\bf v}).\]
Since $\Theta({\bf v})$ is of the same type as $\bf U$, $S$ is NFB by Lemma \ref{nfblemma}.
\end{proof}

\subsection*{Acknowledgement} The authors thank an anonymous referee for helpful comments.

\end{document}